\documentclass[11pt]{article}
\usepackage{amsfonts}
\usepackage{amssymb}
\usepackage{amsthm}
\usepackage{amsmath}
\usepackage{graphicx}
\usepackage{fixltx2e}
\usepackage{gensymb}
\usepackage[numbers]{natbib}

\usepackage{indentfirst}
\usepackage{mathrsfs}
\usepackage{tikz}
\usepackage{graphics}
\usepackage{braket}
\theoremstyle{plain}
\usepackage{enumerate}
\usepackage{upgreek}
\usepackage{siunitx}
\newtheorem{theorem}{\textbf{Theorem}}[section]

\newtheorem{remark}{\textbf{Remark}}[section]

\allowdisplaybreaks[4]

\def\uu{\textbf{u}}
\def\ww{\textbf{w}}
\def\vv{\textbf{v}}
\def\nn{\textbf{n}}
\def\n{\textbf{n}}
\def\R3{\mathbb{R}^3} 
\def\R{\mathbb{R}}
\def\F2o{\overline{F_2}}

\newcommand{\beq}[0]{\begin{equation}}
\newcommand{\eeq}[0]{\end{equation}}
\newcommand{\VV}[0]{{V_{div}}}
\newcommand{\HH}[0]{{G_{div}}}
\newcommand{\ff}[0]{\varphi}

\newcommand{\barf}[0]{\overline{\varphi}}

\newcommand{\dert}[0]{\frac{d}{dt}}
\newcommand{\tmu}[0]{\tilde{\mu}}

\usepackage[hang,flushmargin]{footmisc}
\makeatletter
\def\blfootnote{\gdef\@thefnmark{}\@footnotetext}
\makeatother

\begin{document}
\title{Erratum: ``On the nonlocal Cahn-Hilliard-Brinkman \\ and Cahn-Hilliard-Hele-Shaw systems''\\[1pt] [Comm. Pure Appl. Anal. 15 (2016), 299-317]}

% Place all authors' names in [ ] shown as running head, Leave { } empty
% Please use `and' to connect the last two names if applicable
% Use FirstNameInitial.  MiddleNameInitial. LastName, or only last names of authors if there are too many authors
\author{
{\sc Francesco Della Porta}\footnote{Mathematical Institute, University of Oxford, Oxford OX2 6GG, UK, \textit{dellaporta@maths.ox.ac.uk}},
\
{\sc Maurizio Grasselli}\footnote{Dipartimento di Matematica, Politecnico di Milano, Milano 20133, Italy, \textit{maurizio.grasselli@polimi.it}}
}
% Put your short thanks below. For long thanks/acknowlegements,
%please go to the last acknowlegments section.

% Add corresponding author at the footnote of the first page if it is necessary.
% Plase add $^*$ adjacent to the corresponding author's name on the first page.
% The example shown in this template is if the first author is the corresponding author.
%\thanks{$^*$ Corresponding author: xxxx}

\maketitle
%\centerline{\scshape Francesco Della Porta}
%\medskip
%{\footnotesize
%% please put the address of the first author
% \centerline{Mathematical Institute, University of Oxford}
% %  \centerline{Other lines}
%   \centerline{ Oxford OX2 6GG, UK}
%} % Do not forget to end the {\footnotesize by the sign }
%
%\medskip
%
%\centerline{\scshape Maurizio Grasselli}
%\medskip
%{\footnotesize
% % please put the address of the second  and third author
% \centerline{ Dipartimento di Matematica, Politecnico di Milano}
%%   \centerline{Other lines}
%   \centerline{Milano 20133, Italy}
%}

\bigskip

% The name of the associate editor will be entered by an editorial staff
% "Communicated by the associate editor name" is not needed for special issue.
 \centerline{(Communicated by the associate editor name)}

%The abstract of your paper
%\begin{abstract}
%This is the abstract of your paper and it should not exceed
%\textbf{200} words.
%\end{abstract}
%
%%The title of your section 1
%\section{Introduction}
\bigskip

%\section{Convergence of solutions as $\nu\to 0$}
%\label{s:comparison}
\noindent
\noindent
In this note, we want to highlight and correct an error in \cite[Prop.2.4]{DPG2} which has consequences on the proof of \cite[Thm.6.1]{DPG2}.
Referring to \cite{DPG2} for the notation, the correct statement in \cite[Prop.2.4]{DPG2} is that $\uu \in L^2(0,T;[H^1(\Omega)]^d)$ and not $\uu \in L^2(0,T;\VV)$ as incorrectly written. Therefore we have $\vv(t)=\uu(t)- \uu_\nu(t)\in [H^1(\Omega)]^d$ for almost any $t\in (0,T)$ and the boundary trace of $\vv(t)$ is not necessarily zero. Estimates as the one in \cite[Thm.6.1]{DPG2} are in general difficult to obtain due to the presence of a boundary layer. A common approach to obtain such estimates is to introduce a corrector so that the difference between the solution and the corrector has zero trace (see, e.g., \cite{TemamCorrector}). Here we devise a simpler way to obtain an estimate quite similar to the one reported in \cite[Thm.6.1]{DPG2} without introducing a corrector. However, the order of convergence with respect to $\nu$ is no longer $\frac12$. More precisely, the corrected result reads as follows

\begin{theorem}
\label{t:close 2}
Let (H0), (H2)-(H4), (H8) hold. Suppose $\nu\in (0,1]$, $\eta>0$ constant, $\mathbf{h}=\mathbf{0}$, and
$J$ either be admissible or $J\in W^{2,1}(\R^d)$. Take $\ff_0^\nu,\ff_0\in L^\infty(\Omega)$ and
$$R:=\sup_{\nu> 0}\{\|\ff_0^\nu\|_{L^\infty},\|\ff_0\|_{L^\infty}\}<\infty.$$
Let $(\ff_\nu,\uu_\nu)$ be the unique weak solution to \cite[(1.2)-(1.3)]{DPG2} with initial datum $\ff_0^\nu$,
and $(\ff, \uu)$ be the unique solution to \cite[(1.4)-(1.5)]{DPG2} with initial datum $\ff_0$.
Then, for any given $T>0$, there exists $C=C({R,T,\eta})>0$ such that, for every $\delta\in (0,\frac12)$,
$$\|\ff_\nu(t)-\ff(t)\|_\#^2+\int_0^t\|\uu_\nu(y)-\uu(y)\|^2\,d y\leq
\big(\|\ff_0^\nu-\ff_0\|_\#^2+|\barf_0^\nu-\barf_0|\big)e^{C}+C\nu^{\frac14 - \frac\delta 2},$$
for each $t\in [0,T]$. In particular, if $\ff_0^\nu=\ff_0$, then
$\ff_\nu\to \ff$ in $L^\infty(0,T;V')$ and $\uu_\nu \to \uu$ in $L^2(0,T;\HH)$ as $\nu\to 0.$
\end{theorem}

\begin{proof}
We first notice that the Brinkman equation can be rewritten as follows (see e.g., \cite[eq.(3.40)]{DPG2})
\beq
\label{BrinkGood}
\nu \mathbf{A} \uu_\nu = - \eta \uu_\nu + \mathbf{P}\left(\ff_\nu\nabla\mu_\nu\right),\qquad \text{ a.e. in } \Omega\times(0,T),
\eeq
where $\mathbf{A}$ is the Stokes operator and $\mathbf{P}$ is the Leray projector. Note that the right-hand side belongs to $\HH$ for almost any $t\in(0,T)$. Thus by standard theory (cf. \cite[Chap.IV, Sec.5]{BOY}), we know that $\uu_\nu\in [H^2(\Omega)]^d$. Consequently, we can write
\beq
\label{BrinkGood2}
-\nu (\Delta \uu_\nu,\ww) + {\eta}(\uu_\nu,\ww)=(\ff_\nu\nabla\mu_\nu ,\ww),\qquad \forall\,\ww \in \HH, \text{ a.e. } t\in(0,T).
\eeq
Recalling now  \cite[eq.(3.40)]{DPG2}, we have
\beq
\label{basicidentity}
(\mathbf{P}(\ff_\nu\nabla\mu_\nu),\ww) = ( (\nabla J\ast\ff_\nu)\ff_\nu,\ww)-\frac12(\ff_\nu^2\nabla a,\ww).
\eeq
Therefore, testing \eqref{BrinkGood} with $\nu \mathbf{A}\uu_\nu$ and using Cauchy-Schwartz and Young inequalities, on account of \eqref{basicidentity}, we get
\beq
\label{H2 reg 0}
\frac 12\nu^2 (\mathbf{A} \uu_\nu, \mathbf{A} \uu_\nu) + \eta\nu\|\nabla \uu_\nu\|^2\leq c\|\ff_\nu\|^2_{L^4} \leq C_R,\qquad \text{ a.e. } t\in(0,T),
\eeq
from which we deduce, thanks to \cite[Proposition IV.5.9]{BOY}, that
\beq
\label{H2 reg}
\nu\|\uu_\nu\|_{[H^2]^d} + \sqrt{\nu}\|\uu_\nu\|_{[H^1]^d} \leq C_{R,\eta},\qquad\text{a.e. } t\in(0,T).
\eeq
Here we have also used \cite[Proposition 2.1]{DPG2} for the last inequality in \eqref{H2 reg 0}.

Let us now set $\psi=\ff_\nu-\ff$, $\tilde\mu=\mu_\nu-\mu$ and $\vv=\uu_\nu-\uu$. After subtracting the Darcy equation \cite[(2.9)]{DPG2}
from \eqref{BrinkGood2}, and testing the resulting identity with $\vv$ we get
\[
-\nu ( \Delta \uu_\nu,\vv)  + \|\sqrt{\eta}\vv\|^2= \mathcal{K},
\]
where
\[
 \mathcal{K} :=( \tilde \mu\nabla \ff_\nu + \mu \nabla \psi,\vv)= ( \nabla J\ast\ff_\nu,\psi\vv)+( \nabla J\ast\psi,\ff\vv)-\frac12((\ff_\nu+\ff)\psi\nabla a,\vv).
\]
Integrating by part the viscous term and adding
$-\nu ( \nabla \uu, \nabla \vv )$ to both sides of the resulting identity gives
\beq
\label{DiffNu}
\nu \| \nabla \vv \|^2 + \|\sqrt{\eta}\vv\|^2= \mathcal{K}
-\nu(  \nabla \uu,\nabla \vv)+ \nu\int_{\partial\Omega} \vv^T\nabla\uu_\nu\cdot\nn.
\eeq
Observe that
\beq
\label{trace}
\nu \int_{\partial\Omega} \vv^T\nabla\uu_\nu\cdot\nn = -\nu \int_{\partial\Omega} (\uu^T\nabla\uu_\nu\cdot\n)\leq \|\uu\|_{[L^2(\partial\Omega)]^d} \|\nu\nabla\uu_\nu\|_{[L^2(\partial\Omega)]^d}.
\eeq
%Recalling that
%$$
%\|\uu\|_{L^2(\partial\Omega)}\leq c\|\uu\|_{[H^1]^d},
%$$
On account of the smoothness of the domain $\Omega$, we can use \cite[Prop. 3.8]{Sobolev} and deduce
$$
\nu \|\nabla\uu_\nu\|_{[L^2(\partial\Omega)]^d} \leq \nu \|\nabla\uu_\nu\|_{[H^\delta(\partial\Omega)]^d}\leq c\nu\|\nabla\uu_\nu\|_{[H^{\frac12+\delta}]^d},
$$
with $\delta>0$ arbitrary. Then interpolation yields
$$
\nu\|\nabla\uu_\nu\|_{[H^{\frac12+\delta}]^d} \leq c (\nu\|\uu_\nu\|_{[H^2]^d})^{\frac12+\delta}(\nu\|\uu_\nu\|_{[H^1]^2})^{\frac12-\delta},
$$
for $\delta\in (0,\frac12)$.
Therefore, exploiting \eqref{H2 reg} twice and using a standard trace theorem, from \eqref{trace} we deduce
$$\nu \int_{\partial\Omega} \vv^T\nabla\uu_\nu\cdot\nn \leq C \|\uu\|_{[H^1]^d} (\nu\|\uu_\nu\|_{\VV})^{\frac12-\delta}\leq C\nu^{\frac14-\frac\delta2} \|\uu\|_{[H^1]^d}.
$$
Thus, using also
$$
-\nu(  \nabla \uu,\nabla \vv) \leq \nu\|\nabla\uu\|^2+\nu\|\nabla\vv\|^2,$$
we have that \eqref{DiffNu} becomes
\beq
\label{DiffNu2}
\eta\|\vv\|^2\leq \mathcal{K}
+(\nu+C\nu^{\frac14-\frac\delta2}) \|\uu\|^2_{[H^1]^d}.
\eeq
On the other hand, arguing as in \cite[eq. (5.16)]{DPG2}, we find
\[%begin{align*}
\mathcal{K}
\leq \max{(\|\nabla a\|_{L^\infty},\,\|\nabla J\|_{L^1})}\|\vv\|\bigl(\|\ff_\nu\|_{L^\infty}+\|\ff\|_{L^\infty}\bigr)\|\psi\|
\leq C\|\vv\|\|\psi\|.
\]%end{align*}
Hence, we infer from \eqref{DiffNu2} that ($\nu\leq1$)
$$
\eta\|\vv\|^2\leq C\|\vv\|\|\psi\|+C\nu^{\frac14-\frac\delta2}\|\uu\|^2_{[H^1]^d}
$$
and this implies
\begin{equation}
\label{lavV}
\|\vv\|\leq C\bigl(\|\psi\|+ {\nu^{\frac18-\frac\delta4}}\|\uu\|_{[H^1]^d}\bigr).
\end{equation}
We can now proceed as in the original proof of \cite[Thm 6.1]{DPG2}. More precisely, we have
(cf. \cite[Proof of Prop.2.2]{DPG2})
\begin{equation*}
\frac{1}{2}\dert \|\psi-\bar\psi\|^2_{-1}+ (\tmu,\psi-\bar\psi) = I_1+I_2,
\end{equation*}
where
$$I_1=(\textbf{v}\ff_\nu,\nabla (-\Delta)^{-1}(\psi-\bar\psi)),\qquad I_2=(\textbf{u}\psi,\nabla (-\Delta)^{-1}(\psi-\bar\psi)).$$
Recalling \cite[Proof of Prop. 2.2]{DPG2} we deduce
\begin{align*}
&\frac{1}{2}\dert \|\psi-\bar{\psi}\|^2_{-1}+ \frac{c_0}{4}\|\psi\|^2\leq
N\|\vv\|\|\psi-\bar{\psi}\|_\#+
N\|\psi-\bar{\psi}\|_\#^2
+c\bar{\psi}^2+N|\bar{\psi}|.
\end{align*}
Thus, taking \eqref{lavV} into account, we end up with
\begin{equation}
\label{grou}
\frac{1}{2}\dert \|\psi-\bar{\psi}\|^2_\#+ \frac{c_0}{8}\|\psi\|^2\leq
N\|\psi-\bar{\psi}\|_\#^2
+N|\bar{\psi}|+N\nu^{\frac14-\frac\delta2}\|\uu\|^2_{[H^1]^d}.
\end{equation}
An application of the Gronwall Lemma on $[0,T]$, on account of \cite[Prop. 2.4]{DPG2}, provides
$$\|\ff_\nu(t)-\ff(t)\|_\#^2\leq
\big(\|\ff_0^\nu-\ff_0\|_\#^2+|\barf_0^\nu-\barf_0|\big)e^{C_T}+C_T\nu^{\frac14-\frac\delta2}.
$$
Finally, an integration of \eqref{grou} with respect to time combined with \eqref{lavV} complete the proof.
\end{proof}

\begin{remark}
It is worth pointing out that when the domain $\Omega$ is a torus, then the estimate holds as reported in the original \cite[Thm 2.6]{DPG2}.
Moreover, we observe that the same kind of mistake was made in the proof of \cite[Thm.2.7]{BCG}. Also in that case, the statement has to
be modified according to \eqref{lavV}.
\end{remark}

%For acknowledgements section, please don't number the section, please begin it with \section*{Acknowledgements}
\section*{Acknowledgment} 
{The work of the first author was supported by the Engineering and Physical Sciences Research Council [EP/L015811/1]. The second author is member of the Gruppo Nazionale per l'Analisi Matematica, la Probabilit\`{a} e le loro Applicazioni (GNAMPA) and of the Istituto Nazionale di Alta Matematica (INdAM)}
The authors thank Andrea Giorgini for pointing out the error and contribute to fix it.

% You may incorporate your references as follows in your main tex file.
% Using BibTex is not recommended but can be handled.


\begin{thebibliography}{99}


\bibitem{BCG}(MR3351441)
\newblock S. Bosia, M. Conti and M. Grasselli,
\newblock On the Cahn-Hilliard-Brinkman System,
\newblock \emph{Commun. Math. Sci.}, \textbf{13} (2015), 1541-1567.

\bibitem{BOY}
\newblock F. Boyer and P. Fabrie,
\newblock \emph{Mathematical tools for the study of the incompressible Navier-Stokes equations and related models},
\newblock Appl. Math. Sci. \textbf{183}. Springer, New York, 2013.

\bibitem{DPG2}
\newblock F. Della Porta and M. Grasselli,
\newblock On the nonlocal Cahn-Hilliard-Brinkman and Cahn-Hilliard-Hele-Shaw systems,
\newblock \emph{Comm. Pure Appl. Anal.}, \textbf{15} (2016), 299-317.

%\bibitem{DPG^2} F. Della Porta, A. Giorgini, M. Grasselli, Nonlocal Cahn-Hilliard-Hele-Shaw system with singular
%potentials, \textit{in preparation}.

\bibitem{Sobolev}
\newblock E. Di Nezza, G. Palatucci and E. Valdinoci,
\newblock Hitchhiker's guide to the fractional Sobolev spaces,
\newblock \emph{Bull. Sci. Math.}, \textbf{136} (2012), 521-573.

\bibitem{TemamCorrector}
\newblock J.P. Kelliher, R. Temam and X. Wang,
\newblock Boundary layer associated with the Darcy-Brinkman-Boussinesq model for convection in porous media,
\newblock \emph{Phys.~D}, \textbf{240} (2011), 619-628.

%%:
%\bibitem{A22} (MR2082924)
%     \newblock C.  Wolf,
%     \newblock A mathematical model for the propagation of a hantavirus in structured populations,
%     \newblock \emph{Discrete Continuous Dynam. Systems - B}, \textbf{4} (2004), 1065--1089.

\end{thebibliography}
\end{document}